\documentclass[12pt,a4paper,reqno]{amsart}
\usepackage[T1]{fontenc}
\usepackage[utf8]{inputenc}
\usepackage{amsthm, textcomp, amsmath, amsfonts,  amssymb}
\usepackage[pdftex]{graphicx}
\usepackage{enumerate}
\usepackage{url}
\usepackage{subfig}
\usepackage{comment}
\usepackage[colorlinks]{hyperref}
\usepackage[a4paper,left=2cm,right=2cm]{geometry}
\allowdisplaybreaks[4]

\usepackage{orcidlink}
\author[L. Brandolini]{L. Brandolini \orcidlink{0000-0002-9670-9051}}
\address{\phantom{i}$^1$Dipartimento di Ingegneria Gestionale, dell'Informazione e
della Produzione, Universit{\`a} degli Studi di Bergamo, Viale G. Marconi 5, 24044,
Dalmine BG, Italy}
\email{luca.brandolini@unibg.it}
\email{biancamaria.gariboldi@unibg.it}
\email{giacomo.gigante@unibg.it}
\email{alessandro.monguzzi@unibg.it}
\author[B. Gariboldi]{B. Gariboldi \orcidlink{0000-0001-8714-4135}}

\author[G. Gigante]{G. Gigante$^\dagger$ \orcidlink{0000-0002-1642-679X}}
\author[A. Monguzzi]{A. Monguzzi \orcidlink{0000-0003-3233-5000}} 
\thanks{All the authors are partially supported by the Indam--Gnampa project CUP\textunderscore E53C23001670001 and by the PRIN 2022 project "TIGRECO - TIme-varying signals on Graphs: REal and COmplex methods" funded by the European Union - Next Generation EU, Grant\textunderscore 20227TRY8H, CUP\textunderscore F53D23002630001. A. Monguzzi is also partially supported by the Hellenic Foundation for Research and Innovation
(H.F.R.I.) under the “2nd Call for H.F.R.I. Research Projects to support Faculty Members \&
Researchers” (Project Number: 4662)}
\thanks{\phantom{i}$^\dagger$Corresponding Author}

\begin{document}

%\nocite{*}

\newtheorem{theorem}{Theorem}[section]
\newtheorem{Cor}[theorem]{Corollary}
\newtheorem{lemma}[theorem]{Lemma}
\newtheorem{definition}[theorem]{Definition}
\newtheorem{proposition}[theorem]{Proposition}

\theoremstyle{remark}
\newtheorem{Remark}{Remark}

\newcommand{\nb}[3]{{\colorbox{#2}{\bfseries\sffamily\scriptsize\textcolor{white}{#1}}}
{\textcolor{#2}{\sf\small\textit{#3}}}}
\newcommand{\nota}[1]{\nb{NB}{blue}{#1}}

\title[Single radius spherical cap discrepancy %on compact two-point homogeneous spaces
]
{Single radius spherical cap discrepancy\\ on compact two-point homogeneous spaces}
\begin{abstract}
In this note we study estimates from below of the single radius spherical discrepancy in the setting of compact two-point homogeneous spaces. Namely, given a $d$-dimensional manifold $\mathcal M$ endowed with a distance $\rho$ so that $(\mathcal M, \rho)$ is a two-point homogeneous space and with the Riemannian measure $\mu$, we provide conditions on $r$ such that if $D_r$ denotes the discrepancy of the ball of radius $r$, then, for an absolute constant $C>0$ and for every set of points $\{x_j\}_{j=1}^N$, one has $\int_{\mathcal M} |D_{r}(x)|^2\, d\mu(x)\geqslant C N^{-1-\frac1d}$. The conditions on $r$ that we have depend on the dimension $d$ of the manifold and cannot be achieved when $d \equiv 1 \ ( \operatorname{mod}4)$. Nonetheless, we prove a weaker estimate for such dimensions as well. 
\end{abstract}
\maketitle

\section{Introduction}
In this work we provide some estimates from below  for the single radius spherical discrepancy in the setting of compact two-point homogeneous spaces. To provide some context, let us start by recalling a result of J. Beck. In \cite{Beck84}, but see also \cite[Corollary 24c, pg.182]{BC}, Beck proved that it is not possible to distribute a finite sequence of points on the unit sphere $\mathcal S^d$ so that such distribution of points is regular with respect to spherical caps.
Namely, for $x\in\mathcal S^d$ and $h\in [-1,1]$, let $B(x,h)$ be the spherical cap defined as
\[
B(x,h)=\big\{y\in\mathcal S^d: x\cdot y\geq h\big\}.
\]
Given a finite sequence of points $\{x_j\}_{j=1}^N$, let $\mu$ be the normalized surface measure of $\mathcal S^d$ and let the discrepancy $D(B(x,h))$ of $B(x,h)$ be defined as
\[
D(B(x,h))=\frac{1}{N}\sum_{j=1}^N \chi_{B(x,h)}(x_j)-\mu(B(x,h)).
\]
Then, one has the following.
\begin{theorem}[\cite{Beck84}]
There exists a constant $c>0$ such that for any finite sequence of point $\{x_j\}_{j=1}^N$ it holds
\begin{equation}\label{beck-main}
\int_{-1}^1\int_{\mathcal S^d}|D(B(x,h))|^2\, d\mu(x)dh\geq c N^{-1-\frac 1d}.
\end{equation}
\end{theorem}
%Such seminal result of Beck inspired several subsequent works and similar estimates have been proved both considering different ambient spaces other than $\mathcal S^d$, such as the $d$-dimensional torus \cite{M} or two-point homogeneous spaces \cite{Skriganov, twopoint}, and/or considering  different testing family other than the one of spherical caps (e.g., rectangles, discs, convex sets). See, e.g., \cite{BT} or \cite{Panorama}. See also \cite{BGGM2} for a version of Roth's theorem in a discrete setting.    

An important feature to point out in Beck's estimate is that in the left-hand side of \eqref{beck-main} one averages over all the possible values of $h$. 
It is still unclear, and a matter of investigation, if such averaging ploy is necessary or not to obtain large discrepancy estimates. In regard of this matter, 
H. Montgomery \cite{Mon89,M} proved, for instance, that in order to have large discrepancy for discs in the two dimensional torus, it is enough to average on two discs only, one of radius $1/2$ and the other of radius $1/4$. His strategy is based on an inequality on exponential sums, now known as Cassels-Montgomery inequality \cite[Theorem 5.12]{M}, which can be naturally extended to higher dimensions, and has been proved to hold for eigenfunctions of the Laplacian on compact Riemannian manifolds \cite{BGG}. This strategy has proven to be very versatile, allowing to obtain several estimates from below of the discrepancy with respect to different testing families such as rectangles or convex sets \cite{BT,Panorama, BGGM2, BCT23}, or in different ambient spaces other than the $d$-dimensional torus, such as
the sphere itself or compact two-point homogeneous spaces \cite{BMS, twopoint}.

In particular, in \cite{twopoint} the first three authors of this paper proved a similar ``two radius'' estimate for the discrepancy with respect to balls in the setting of compact two-point homogeneous spaces. However, before stating such result it is necessary to recall M. Skriganov's generalization of Beck's result \cite{Skriganov}.

Let $\mathcal M$ be a compact $d$-dimensional Riemannian manifold endowed with metric $\rho$ so that $(\mathcal M,\rho)$ is a two-point homogeneous space (the definition of such spaces is recalled later). Let $\mu$ be the normalized Riemannian measure on $\mathcal M$ and let $B_r(x)=\{y\in\mathcal M: \rho(x,y)<r\}$. 
For a given finite sequence of points $\left\{  x_{j}\right\}  _{j=1}^{N}\subseteq
\mathcal{M}$ and positive weights $\left\{  a_{j}\right\}  _{j=1}^{N}$ such
that $a_{1}+a_{2}+\cdots+a_{N}=1$ we define the discrepancy of the ball
$B_{r}(x)$ by
\begin{equation}
D_{r}(x)=\sum_{j=1}^{N}a_{j}\chi_{B_{r}(x)}(x_{j})-\mu(B_{r}(x)).
\label{Discrepanza M}%
\end{equation}
In the case of equal weights, it is proved in \cite[Theorem 2.2]{Skriganov} that if $\eta$ is a positive, locally integrable
function on $(0,\pi)$ which satisfies a suitable integrability condition, then
 there exists $c>0$ such that for every finite sequence
 of $N$ points the estimate
\begin{equation}\label{skriganov-main}
\int_{0}^{\pi}\int_{\mathcal{M}}\left\vert D_{r}(x)\right\vert ^{2}d\mu
(x)\eta(r)dr\geq cN^{-1-\frac{1}{d}}
\end{equation}
holds.
Such estimate is sharp (see, e.g.,\cite[Corollary 8.2]{BCCGT}).  In fact, Skriganov \cite[Corollary
2.1]{Skriganov} proved that for well-distributed optimal cubature formulas
one has
\begin{equation}
\int_{0}^{\pi}\int_{\mathcal{M}}\left\vert D_{r}(x)\right\vert ^{2}d\mu
(x)\eta(r)dr\leq cN^{-1-\frac{1}{d}}. \label{stima cubatura alto}%
\end{equation}
The existence of such cubature formulas for the sphere has been proved by A. Bondarenko, D. Radchenko and  M. Vyazovska in \cite{BRV}, whereas for more general Riemannian manifolds it has been proved in \cite{EEGGT,GG}.

Going back to the issue of single radius estimates, notice that, similarly to Beck's result, also in the left-hand side of \eqref{skriganov-main} one is averaging over all the possible radii. However, the following two radii result is proved in \cite{twopoint}. If $d\not\equiv\, 1 (\operatorname{mod} 4)$, for any $0<r<\frac\pi2$, one has the sharp estimate
\[
\int_{\mathcal M} \left(|D_r(x)|^2+|D_{2r}(x)|^2\right) d\mu(x)\geq cN^{-1-\frac 1d}.
\]
If $d\equiv\, 1 (\operatorname{mod} 4)$ the technique used in \cite{twopoint}, which goes back to Montomery's aforementioned works,  fails. It fails also if one is willing to average on any finite number of radii.
At the cost of losing the sharpness, a single radius estimate, which holds for any dimension, is also proved in \cite{twopoint}. Namely, for every $\varepsilon>0$ and for almost every $0<r<\pi$ there exists a constant $c>0$ such that, for every finite sequence $\{x_j\}_{j=1}^N\subseteq\mathcal M$, one has
\begin{equation}\label{BGG-single}
\int_{\mathcal M} |D_r(x)|^2\, d\mu(x)\geq c N^{-1-\frac 3d-\varepsilon}.
\end{equation}
In a recent work of D. Bilyk, M. Mastrianni and S. Steinerberger  a sharp single radius estimate is proved for some suitable radii in the case of $\mathcal M=\mathcal S^d$ with $d\not\equiv 1\, (\operatorname{mod} 4)$. Precisely, in \cite[Theorem 5]{BMS} it is proved that if $\mathcal M=\mathcal S^d$ with $d\not\equiv 1\, (\operatorname{mod} 4)$ and if $\cos r$ is $(d+1)/2-$gegenbadly approximable, then
\[
\int_{\mathcal M} |D_r(x)|^2 d\mu(x)\geq cN^{-1-\frac 1d}.
\]
We recall that $x\in (-1,1)$ is $\lambda$-gegenbadly approximable, $\lambda>0$, if there exists a constant $c_x>0$ such that, for all $m\in\mathbb N$, 
\[
|P^{\lambda}_m(x)|\geq c_x m^{\lambda-1},
\]
where $P^{\lambda}_m$ denotes the Gegenbauer polynomial of degree $m$ (see \cite[Section 1]{BMS}). In particular, a necessary and sufficient condition to check if a number is $\lambda$-gegenbadly approximable or not is also given in \cite{BMS}. See also Corollary \ref{gegenbadly} here. 

 Our first result is in the same spirit of the one by Bilyk--Mastrianni--Steinerberger. However, before stating it, we need to precisely introduce the setting in which we are working. A $d$-dimensional Riemannian manifold $\mathcal{M}$ with distance $\rho$ is
said to be a two-point homogeneous space if given four points $x_{1}%
,x_{2},y_{1},y_{2}\in\mathcal{M}$ such that $\rho(x_{1},y_{1})=\rho
(x_{2},y_{2})$, then there exists an isometry $g$ of $\mathcal{M}$ such that
$gx_{1}=x_{2}$ and $gy_{1}=y_{2}$. Compact connected two-point homogeneous spaces have been completely characterized by H. Wang \cite{Wang}. Namely,  it turns out that $\mathcal{M}$ is isometric to one of the following compact rank
$1$ symmetric spaces:

\begin{itemize}
\item[(i)] the Euclidean sphere $\mathcal S^{d}=SO(d+1) /SO(d)\times\{1\}$,
$d\geqslant1$;

\item[(ii)] the real projective space $P^{n}( \mathbb{R}) =O(n+1) /O( n)
\times O( 1) $, $n\geqslant2$;

\item[(iii)] the complex projective space $P^{n}( \mathbb{C}) =U(n+1) /U( n)
\times U( 1) $, $n\geqslant2$;

\item[(iv)] the quaternionic projective space $P^{n}( \mathbb{H})=Sp(n+1)/Sp(
n) \times Sp( 1) $, $n\geqslant2$;

\item[(v)] the octonionic projective plane $P^{2}( \mathbb{O}) $.
\end{itemize}

From now on we will always assume that $\mathcal{M}$ is one of the above symmetric spaces  with a metric $\rho$ normalized so that $\mathrm{diam}(\mathcal{M}%
)=\pi$ and the Riemannian measure $\mu$ normalized so that $\mu(\mathcal M)=1$. If $d$ denotes the real dimension of $P^{n}(\mathbb{F})$, then $d=nd_{0}$, where $d_{0}=1,2,4,8$
according to the real dimension of $\mathbb{F}=\mathbb{R}$, $\mathbb{C}$,
$\mathbb{H}$ and $\mathbb{O}$ respectively. In the case of $\mathcal S^{d}$ it will be
convenient to set $d_{0}=d$. See \cite[pp. 176-178]{Gangolli}, see also
\cite{H}, \cite{Skriganov} and \cite{Wolf}.
Recall that $D_r(x)$ denotes the discrepancy for balls defined as in \eqref{Discrepanza M}. The following is our extension to compact two point homogeneous spaces of the result in \cite{BMS}.

\begin{theorem}
\label{Thm 1}
Let $p$ and $q$ be coprime integers, $0<p<q$, such that
\begin{equation}
\label{mod4}
\frac{d+d_0+2}4p-\frac{d-1}4q\notin\mathbb Z.
\end{equation}
Then, there exists a constant $C>0$ such that for every set of points
$\left\{  x_{j}\right\}  _{j=1}^{N}\subseteq\mathcal{M}$ and positive weights
$\left\{  a_{j}\right\}  _{j=1}^{N}$ satisfying $a_{1}+a_{2}+\cdots+ a_{N}=1$
we have
\[
\int_{\mathcal{M}}\left\vert D_{p\pi/q}( x) \right\vert ^{2}d\mu( x)  \geq
CN^{-1-\frac{1}{d}}.
\]
\end{theorem}

We point out that condition \eqref{mod4} can always be achieved for particular choices of $p$ and $q$, except when $\mathcal M$ is the Euclidean sphere or the real projective space with dimension $d \equiv 1 \ ( \operatorname{mod}4)$. Observe that in all the other cases we are interested in the dimension is even, and therefore $d \not\equiv 1 \ ( \operatorname{mod}4)$.  In this case, indeed, condition \eqref{mod4} is simply satisfied by choosing, for example, $p\in4 \mathbb Z$ and $q$ odd. Again this result is optimal in view of Corollary 8.2 in \cite{BCCGT}. Moreover, with the same technique used by Skriganov to prove the estimate from above for
cubature formulas (\ref{stima cubatura alto}), it is possible to prove a uniform estimate in the radius $r$,
\begin{equation*}
\int_{\mathcal{M}}\left\vert D_{r}(x)\right\vert ^{2}d\mu(x)\leq
cN^{-1-\frac{1}{d}}, %\label{dallalto}%
\end{equation*}
for suitable choices of weights $\{a_j\}$ and points $\{x_j\}$. See \cite[Theorem 15]{twopoint} for a precise statement.

\color{black}

When $d\equiv1\,\left(  \operatorname{mod}4\right)$ the technique we use to prove Theorem \ref{Thm 1} fails since the estimate of Theorem \ref{jacobadly} does not hold. %for the same reason it fails in the case of two or more radii.
In the case of the $d$-dimensional torus $\mathbb{T}^d$ it is known that for such values of $d$  the discrepancy can actually be a bit
smaller than the expected value $N^{-1-\frac{1}{d}}$. See \cite[Theorem
3.1]{PS}. See also \cite{BCGT}, \cite{KSS} and \cite{PSido}. 
% In \cite[Theorem 2]{twopoint} it is proved that
% for every $\varepsilon>0$ and for almost every $0<r<\pi$ there is
% a constant $C>0$ such that for every set of points $\left\{  x_{j}\right\}
% _{j=1}^{N}\subset\mathcal{M}$ and positive weights $\left\{  a_{j}\right\}
% _{j=1}^{N}$ such that $a_{1}+a_{2}+\cdots+ a_{N}=1$,
% \begin{equation*}
% \int_{\mathcal{M}}\left\vert D_{r}(x)\right\vert ^{2}d\mu(x)\geq
% CN^{-1-\frac{3}{d}-\varepsilon}.
% \end{equation*}
We remind that when $d\equiv 1\left(\operatorname{mod}4\right)$,  $\mathcal M$ can only be the Euclidean sphere or the real projective space.
In this case we can prove the following result (cf. with \cite[Theorem 2]{twopoint}).

\begin{theorem}\label{d1mod4}
Let $d\equiv1\,(  \operatorname{mod}4)$. Let $\{q_n\}$ be the sequence of primes in increasing order and let $\{p_n\}$ be a sequence of positive integers such that for some $\delta>0$ and for every $n$ we have $\delta\leq p_n/q_n\leq 1-\delta$. Then,  for every $N\geq 3$, for every choice of points $\{x_j\}_{j=1}^N$ and positive weights $\{a_j\}_{j=1}^N$ with $\sum_{j=1}^N a_j=1$ there exists $n\leq c\log N/\log\log N$ such that   
\[
\int_{\mathcal{M}}\left\vert D_{p_n\pi/q_n}( x) \right\vert ^{2}d\mu( x)  \geq
CN^{-1-\frac{1}{d}}\frac{\log\log N}{\log^4 N}.
\]
%Furthermore, if $N$ is sufficiently large, then $n\leq \log N/(\log\log N)^{1-\varepsilon}$.
\end{theorem}

In the next section we recall some preliminaries necessary to work in two-point homogeneous spaces, whereas in Section \ref{section-main-results} we prove our main results.

\medskip

With the notation $A\approx B$, we mean the fact that there exist two constant $c_1$ and $c_2$ independent of the involved variables such that $c_1 A\leq B\leq c_2 A$.

\section{Harmonic analysis preliminaries in two-point homogeneous spaces}\label{section-2}
The preliminaries contained in this section are essentially taken from the literature and from \cite{twopoint}. Hence, we omit most of the proofs and we invite the reader to refer to \cite{twopoint} and the references therein.

In the following, to keep notation
simple, we will use
\[
a =\frac{d-2}{2}, \ \quad\ b =\frac{d_{0}-2}{2}.
\] 
Recall that if $\mathcal M= P^n(\mathbb F)$, then $d=nd_0$, where $d_0=1,2,4,8$ according to the real dimension of $\mathbb F=\mathbb R, \mathbb C,\mathbb H$ and $\mathbb O$ respectively. Instead, if $\mathcal M=\mathcal S^d$ it is convenient to set $d=d_0$.

If $o$ is a fixed point in $\mathcal{M}$, then $\mathcal{M}$ can be identified
with the homogeneous space $G/K$, where $G$ is the group of isometries of
$\mathcal{M}$ and $K$ is the stabilizer of $o$. We will also identify
functions $F(x)$ on $\mathcal{M}$ with right $K$-invariant functions $f(g)$ on
$G$ by setting $f(g)=F(x)$ when $go=x$. If
%$\sigma$ denotes the Haar measure
%on $G$ (normalized so that $\sigma(G)=1$) and
$\mu$ is the Riemannian measure on $\mathcal{M}$ normalized so that
$\mu(\mathcal{M})=1$,
%since
then $\mu$ is invariant under the action of $G$,
%it follows that%
%\begin{align*}
%\int_{\mathcal{M}}F(x)d\mu(x)  &  =\int_{G}\int_{\mathcal{M}}F(gx)d\mu
%(x)d\sigma(g)\\
%&  =\int_{\mathcal{M}}\int_{G}F(gx)d\sigma(g)d\mu(x)=\int_{\mathcal{M}}%
%\int_{G}F(go)d\sigma(g)d\mu(x)\\
%&  =\int_{G}F(go)d\sigma(g)=\int_{G}f(g)d\sigma(g).
%\end{align*}
in other words, for every $g\in G$,
\[
\int_{\mathcal{M}}F(gx)d\mu(x)=\int_{\mathcal{M}}F(x)d\mu(x)
\]

\begin{definition}
A function $F$ on $\mathcal{M}$ is a zonal function (with respect to $o$) if for every
$x\in\mathcal{M}$ and every $k\in K$ we have $F( kx) =F( x) $.
We will say that $F$ is zonal with respect to $y$ if $F(gx)$ is a zonal function (with respect to $o$) and $y=go$.
\end{definition}

\begin{lemma}
\label{Lemma misura}Let $F$ be a zonal function. Then $F( x) $ depends only on
$\rho( x,o) $. Furthermore, defining $F_{0}$ so that $F( x) =F_{0}( \rho( x,o)
) $ we have%
\begin{equation}
\int_{\mathcal{M}}F( x) d\mu( x) =\int_{0}^{\pi}F_{0}( r) A( r) dr,
\label{IntegraleZonale}%
\end{equation}
where%
\[
A( r) =c( a,b) \left(  \sin\frac{r}{2}\right)  ^{2a+1}\left(  \cos\frac{r}%
{2}\right)  ^{2b+1}%
\]
and
\[
c( a,b) =\left(  \int_{0}^{\pi}\left(  \sin\frac{r}{2}\right)  ^{2a+1}\left(
\cos\frac{r}{2}\right)  ^{2b+1}dr\right)  ^{-1} =\frac{\Gamma(a+b+2)}%
{\Gamma(a+1)\Gamma(b+1)}.
\]
% In particular if%
% \[
% B_{r}( x) =\{ y\in\mathcal{M}:\rho( x,y) <r\} ,
% \]
% there exist two constants $c_{1}( d,d_{0}) $ and $c_{2}( d,d_{0}) $ such that
% for every $r\in[ 0,\pi] $%
% \begin{equation}
% c_{1}( d,d_{0}) r^{d}\leq\mu( B_{r}( x) ) \leq c_{2}( d,d_{0})
% r^{d}. \label{Volume palle}%
% \end{equation}
% Moreover if $k\geqslant1$, then a ball of radius $kr$ can be covered by
% $\frac{c_{2}( d,d_{0}) }{c_{1}( d,d_{0}) }( 2(k+1)) ^{d}$ balls of radius $r$.
\end{lemma}

% \begin{proof}
% Let $x,y\in\mathcal{M}$ such that $\rho( x,o) =\rho( y,o) $. Since
% $\mathcal{M}$ is two-point homogeneous there exists $g\in G$ such that $gx=y$
% and $go=o$. Thus, $g\in K$ and $F( y) =F( gx) =F( x) $. Equation
% (\ref{IntegraleZonale}) follows from (4.17) in \cite{Gangolli}. The estimate
% (\ref{Volume palle}) is an immediate consequence of (\ref{IntegraleZonale}).
% The last assertion is an estimate of the maximum number of disjoint balls of
% radius $r/2$ that can fit in a ball of radius $(k+1)r$.
% \end{proof}

Let $\Delta$ be the Laplace-Beltrami operator on $\mathcal{M}$, let
$\lambda_{0},\lambda_{1},\ldots,$ be the distinct eigenvalues of $-\Delta$
arranged in increasing order, let $\mathcal{H}_{m}$ be the eigenspace
associated with the eigenvalue $\lambda_{m}$, and let $d_{m}$ its dimension.
%It follows from Weyl's estimate that
%\begin{equation}
%d_{m}\sim m^{d-1}. \label{Weyl}%
%\end{equation}
It is well known that%
\begin{equation}
L^{2}(\mathcal{M})=%
%TCIMACRO{\dbigoplus _{m=0}^{+\infty}}%
%BeginExpansion
{\displaystyle\bigoplus_{m=0}^{+\infty}}
%EndExpansion
\mathcal{H}_{m}. \label{Decomp L2}%
\end{equation}
If $F(x)=F_{0}(\rho(x,o))$ is a zonal function on $\mathcal{M}$, then%
\begin{equation}
\Delta F(x)=\left.  \frac{1}{A(t)}\frac{d}{dt}\left(  A(t)\frac{d}{dt}%
F_{0}(t)\right)  \right\vert _{t=\rho(x,o)} \label{Radial Laplace Beltrami}%
\end{equation}
(see (4.16) in \cite{Gangolli}).

\begin{definition}
The zonal spherical function of degree $m\in\mathbb{N}$ with pole
$x\in\mathcal{M}$ is the unique function $Z_{x}^{m}\in\mathcal{H}_{m}$, given
by the Riesz representation theorem, such that for every $Y\in\mathcal{H}_{m}$%
\[
Y( x) =\int_{\mathcal{M}}Y( y) Z_{x}^{m}( y) d\mu( y) .
\]

\end{definition}

The next lemma summarizes the main properties of zonal functions. The case $\mathcal M=\mathcal S^d$ is discussed in detail in \cite{SW}, whereas we refer to \cite[Lemma 6]{twopoint} for the general case.

\begin{lemma}\label{proprieta}
\begin{itemize}

\item[i)] If $Y_{m}^{1},\ldots,Y_{m}^{d_{m}}$ is an orthonormal basis of
$\mathcal{H}_{m}\subset L^{2}( \mathcal{M}) $, then%
\[
Z_{x}^{m}( y) =\sum_{\ell=1}^{d_{m}}\overline{Y_{m}^{\ell}( x) }Y_{m}^{\ell}(
y) .
\]

\item[ii)] $Z_{x}^{m}$ is real valued and $Z_{x}^{m}( y) =Z_{y}^{m}( x) .$

\item[iii)] If $g\in G$, then $Z_{gx}^{m}( gy) =Z_{x}^{m}( y) .$

\item[iv)] $\forall x\in\mathcal M, \ \|Z_x^m\|_\infty= Z_{x}^{m}( x) =\Vert Z_{x}^{m}\Vert_{2}^{2}=d_{m}.$ 

\item[v)] $Z_{o}^{m}( x) $ is a zonal function and
\begin{equation}
Z_{o}^{m}( x) =\frac{d_{m}}{P_{m}^{a,b}( 1) }P_{m}^{a,b}( \cos( \rho( x,o) ) )
\label{Jacobi}%
\end{equation}
where $P_{m}^{a,b}$ are the Jacobi polynomials.

\item[vi)] $\big\{  d_{m}^{-1/2}Z_{o}^{m}\big\}  _{m=0}^{+\infty}$ is an
orthonormal basis of the subspace of $L^{2}(\mathcal{M})$ of zonal functions.

\item[vii)] Let $\mathbb{P}_{m}$ denote the orthogonal projection of $L^{2}(
\mathcal{M}) $ onto $\mathcal{H}_{m}$. Then for every zonal function $f$,
\[
\mathbb P_{m}f(x)=d_m^{-1}\int_{\mathcal M}f(y)Z_o^m(y)d\mu(y)Z_o^m(x).
\]

\item[viii)] $\lambda_{m}=m(m+a+b+1)$.

\item[ix)] We have
\begin{align*}
d_{m}  &  =  (2m+a+b+1) \frac{\Gamma(  b+1)  }{\Gamma(
a+1)  \Gamma(  a+b+2)  }\frac{\Gamma(  m+a+b+1)
}{\Gamma(  m+b+1)  }\frac{\Gamma(  m+a+1)  }%
{\Gamma(  m+1)  } \approx m^{d-1}.%
\end{align*}
\end{itemize}
\end{lemma}

\section{Proof of the main results}\label{section-main-results}
In this section we prove our main results. In order to do so we need some preliminary results taken from \cite{twopoint}. In the following lemma we obtain an identity for the $L^2$ discrepancy which separates the contribution of the geometry of the balls $B_r(x)$ from the one of the sequence $\{x_j\}_{j=1}^N$ and of the weights $\{a_j\}_{j=1}^N$. 
\begin{lemma}
\label{Lemma Quadrato Discrep}Let $D_{r}( x) $ be as in (\ref{Discrepanza M}).
Then%
\[
\int_{\mathcal{M}}\left\vert D_{r}( x) \right\vert ^{2}d\mu( x) =\sum
_{m=1}^{+\infty}\sum_{\ell=1}^{d_{m}}\left\vert \sum_{j=1}^{N}a_{j}Y_{m}%
^{\ell}( x_{j}) \right\vert ^{2}d_{m}^{-2}\left\vert \int_{B_{r}( o) }%
Z_{o}^{m}( y) d\mu( y) \right\vert ^{2}.
\]

\end{lemma}
\begin{proof}
By (\ref{Decomp L2}) we have%
\[
\int_{\mathcal{M}}\left\vert D_{r}(x)\right\vert ^{2}d\mu(x)=\sum
_{m=0}^{+\infty}\int_{\mathcal{M}}|\mathbb{P}_{m}D_{r}(x)|^{2}d\mu(x)
\]
Since $o$ is arbitrary and $\chi_{B_{r}(x_{j})}$ is zonal with respect to $x_j$, applying (vii) of Lemma \ref{proprieta} with $o$ substituted by $x_j$, we have
\[\mathbb{P}_{m}\left(\chi_{B_{r}(\cdot)}(x_{j})\right)(x)=\mathbb{P}_{m}\left(\chi_{B_{r}(x_j)}\right)(x)=d_m^{-1}\int_{\mathcal M} \chi_{B_{r}(x_j)}(y)Z_{x_j}^m(y)d\mu(y)Z_{x_j}^m(x).\]
Overall, applying (iii) of Lemma \ref{proprieta} we obtain%
\[
\mathbb{P}_{m}D_{r}(x)=\sum_{j=1}^{N}a_{j}d_{m}^{-1}\int_{B_{r}(o)}Z_{o}%
^{m}(y)d\mu(y)Z_{x_{j}}^{m}(x)-\delta_{0}(m)\mu(B_{r}(o)),
\]
where $\delta_0(m)$ is the Kronecker delta. In particular $\mathbb{P}_{0}D_{r}(x)=0$ and for $m>0$%
\begin{align*}
\mathbb{P}_{m}D_{r}(x)  &  =d_{m}^{-1}\sum_{j=1}^{N}a_{j}\int_{B_{r}(o)}%
Z_{o}^{m}(y)d\mu(y)Z_{x_{j}}^{m}(x)\\
&  =d_{m}^{-1}\sum_{\ell=1}^{d_{m}}\left(  \sum_{j=1}^{N}a_{j}\overline
{Y_{m}^{\ell}(x_{j})}\right)  \int_{B_{r}(o)}Z_{o}^{m}(y)d\mu(y)Y_{m}^{\ell
}(x).
\end{align*}
Finally,
\[
\int_{\mathcal{M}}|D_{r}(x)|^{2}d\mu(x)=\sum_{m=1}^{+\infty}\sum_{\ell
=1}^{d_{m}}\left\vert \sum_{j=1}^{N}a_{j}Y_{m}^{\ell}(x_{j})\right\vert
^{2}d_{m}^{-2}\left\vert \int_{B_{r}(o)}Z_{o}^{m}(y)d\mu(y)\right\vert ^{2}.
\]

\end{proof}

It is clear from the above lemma that we now need to estimate the quantities%
\begin{equation}
\label{twoquant}
\sum_{\ell=1}^{d_{m}}\left\vert \sum_{j=1}^{N}a_{j}Y_{m}^{\ell}(x_{j}%
)\right\vert ^{2} \quad\text{and } \quad d_{m}^{-2}\left\vert \int_{B_{r}%
(o)}Z_{o}^{m}(y)d\mu(y)\right\vert ^{2}.
\end{equation}
The first of these quantities is controlled with a Cassels-Montgomery-type estimate in the following proposition.

\begin{proposition}
\label{Prop Cassels} There exist $C_{0},C_{1}>0$ such that
for every $L\geq M>0$ we have%
\[
\sum_{m=M}^{L}\sum_{\ell=1}^{d_{m}}\left\vert \sum_{j=1}^{N}a_{j}%
Y_{m}^{\ell}(x_{j})\right\vert ^{2}\geq C_{1}\sum_{j=1}^{N}a_{j}^{2}%
L^{d}-C_{0}M^d.
\]

\end{proposition}

\begin{proof}
By the Cassels-Montgomery inequality for manifolds (see \cite[Theorem 1]{BGG})
along with the fact that \[\sum_{m=0}^{L}d_{m}\approx\sum_{m=0}^{L}%
m^{d-1}\approx L^{d},\] we have%
\[
\sum_{m=0}^{L}\sum_{\ell=1}^{d_{m}}\left\vert \sum_{j=1}^{N}a_{j}Y_{m}^{\ell
}(x_{j})\right\vert ^{2}\geq C_{1}\sum_{j=1}^{N}a_{j}^{2}L^{d}.
\]
Since, by Lemma \ref{proprieta} (iv) $Z_{x_j}^m(x_k)\leq d_m$,
\begin{align*}
\sum_{m=0}^{M-1}\sum_{\ell=1}^{d_{m}}\left\vert \sum_{j=1}^{N}a_{j}%
Y_{m}^{\ell}(x_{j})\right\vert ^{2}  & = \sum
_{m=0}^{M-1}\sum_{j,k=1}^{N}a_{j}a_k Z^m_{x_j}(x_k)\leq \sum_{m=0}^{M-1} d_m\leq C_0 M^d,
\end{align*}
and the proposition follows. 
\end{proof}
Moving now to the second quantity in \eqref{twoquant}, in the next lemma we obtain an identity for the integral of zonal spherical functions on a ball.
\begin{lemma}
\label{Trasf bolla}For any $0\leq r\leq\pi$ and for any
$m\geq1$ we have%
\[
\int_{B_{r}(o)}Z_{o}^{m}(x)d\mu(x)=\frac{c( a,b) d_{m}}{mP_{m}^{a,b}%
(1)}P_{m-1}^{a+1,b+1}(\cos r)\left(  \sin  \frac{r}{2}  \right)
^{2a+2}\left(  \cos  \frac{r}{2}  \right)  ^{2b+2}.
\]

\end{lemma}

\begin{proof}
By (\ref{IntegraleZonale}) and (\ref{Jacobi}), 
\begin{align*}
\int_{B_{r}(o)}Z_{o}^{m}(x)d\mu(x)  &  =\frac{d_{m}}{P_{m}^{a,b}(1)}\int
_{0}^{r}P_{m}^{a,b}(\cos t)A(t)dt\\
&  =\frac{c(a,b)d_{m}}{2^{a+b+1}P_{m}^{a,b}(1)}\int_{\cos r}^{1}P_{m}%
^{a,b}(x)(1-x)^{a}(1+x)^{b}dx,
\end{align*}
and the thesis follows applying  Rodrigues' formula (see \cite[(4.3.1)]{szego})
\[P_{m}^{a,b}(x)(1-x)^{a}(1+x)^{b}=-\frac{1}{2m}\frac{d}%
{dx}\left(  P_{m-1}^{a+1,b+1}(x)(1-x)^{a+1}(1+x)^{b+1}\right).
\]
\end{proof}

By Lemma \ref{Lemma Quadrato Discrep} we need to estimate 
\begin{equation}\label{int}
\left\vert \int_{B_{r}(o)}Z_{o}^{m}(y)d\mu(y)\right\vert ^{2}.
\end{equation}
However, it is clear from the previous lemma that this quantity vanishes when $\cos r$ is a zero of the Jacobi polynomial $P_{m-1}^{a+1,b+1}$. On the other hand, since (see \cite[formula (4.3.3)]{szego})
\[\int_0^\pi \!\!P_m^{\alpha,\beta}(\cos r)^2 \left(\sin \frac r2 \right)^{2\alpha+1}\!\!\left( \cos \frac r2\right)^{2\beta+1}\!\!dr=\frac{\Gamma(m+\alpha+1)\Gamma(m+\beta+1)}{{(2m+\alpha+\beta+1)}\Gamma(m+1)\Gamma(m+\alpha+\beta+1)}\geq \frac{C}{m},\]
we would expect that, on average, $|P_m^{\alpha,\beta}(\cos r)|\geq Cm^{-1/2}$. The following theorem identifies the values of $r$ for which this relation holds. 

\begin{comment}
Lemma \ref{coeff_bolla} shows that this quantity may vanish due to the zeroes
of the Bessel functions. Our next lemma shows that at least when
$d\not \equiv 1\,( \operatorname{mod}4) $ one can overcome this obstruction
using two balls of different radii. See the proof of Theorem 2 in Chapter 6 in \cite{Mo}.
\end{comment}

\begin{theorem}
\label{jacobadly}
Let $\alpha>-1$ and $\beta>-1$, and call $\gamma=(\alpha+\beta+1)/2$ and $\delta=-(2\alpha-1)/4$. Let $r\in\mathbb R\setminus \pi\mathbb Z$. There exist positive constants $C$ and $m_{0}$ such
that for $m\geq m_{0}$,%
\[
\left|P_m^{\alpha,\beta}(\cos r)\right|\geq Cm^{-\frac12}
\]
if and only if $r/\pi=p/q\in\mathbb Q\setminus\mathbb Z$, with $p$ and $q$ coprime, and 
\begin{equation}
\label{cccond}
\gamma p+\delta q\notin\mathbb Z.
\end{equation}
\end{theorem}

Notice that if $r\in\pi\mathbb Z$, then the situation is somewhat particular. More precisely $|P_m^{\alpha,\beta}(\cos r)|=|P_m^{\alpha,\beta}(\pm1)|$ and this is $|{m+\alpha \choose m}|\approx m^\alpha$ or $|(-1)^m{m+\beta \choose m}|\approx m^\beta$.

\begin{proof}
Remember that for $r\notin \pi\mathbb Z$ (see \cite[Theorem 8.21.8]{szego})
\begin{align*}
&P_{m}^{\alpha,\beta}(\cos r)\\
&=m^{-\frac{1}{2}}\pi^{-\frac{1}{2}}\left(\sin \frac{r}{2} \right)^{-\alpha-\frac{1}{2}}\left(\cos\frac{r}{2} \right)^{-\beta-\frac{1}{2}}\cos\left(\Big(m+\frac{\alpha+\beta+1}{2} \Big)r-\frac{2\alpha+1}{4}\pi \right)+O(m^{-\frac{3}{2}}).
\end{align*}
We are now looking for values of $r$ for which there exists $\eta>0$ and $m_0\ge1$ such that
\begin{equation}
\label{coseno1}
\left|\cos\left(\Big(m+\frac{\alpha+\beta+1}{2}\Big)r-\frac{2\alpha+1}{4}\pi \right)\right|\geq \eta
\end{equation}
for all integer values of $m\geq m_0$. If $r/\pi\notin \mathbb Q$ then $\{mr:m\geq m_0\}$ is dense $\text{mod }\pi$ in $[0,\pi]$ and condition \eqref{coseno1} cannot be achieved. Assume now that $r/\pi=p/q\in\mathbb Q$ with $p$ and $q$ coprime. Then condition \eqref{coseno1} is equivalent to asking that, for all $m=1,\ldots,q$,
\[
\left(m+\frac{\alpha+\beta+1}{2}\right)\frac pq-\frac{2\alpha+1}{4}+\frac 12\notin\mathbb Z,
\]
(notice that the above condition is $q$-periodic in $m$) or equivalently
\begin{equation}\label{cond11}
Q=pm+\frac{\alpha+\beta+1}{2}p-\frac{2\alpha-1}4q=pm+\gamma p+\delta q \quad \text{is not a multiple of } q.
\end{equation}
Also, \eqref{cond11}    is equivalent to require that
\begin{equation*}
H=\gamma p+\delta q\notin\mathbb Z.
\end{equation*}
Indeed, if $H$ is not an integer, then $Q$ is not an integer either and \eqref{cond11} holds. Conversely, assume that $H$ is an integer. Then, since $p$ and $q$ are coprime, the equation $pm+qj=-H$ has integer solutions, and \eqref{cond11} does not hold. 
\end{proof}

According to the values of $\gamma$ and $\delta$, it may be the case that condition \eqref{cccond} holds for any possible choice of coprime $p$ and $q$, for particular values of $p$ and $q$ or for no values of $p$ and $q$. The following proposition shows all the  possible different cases.

\begin{proposition}
    \label{jacobcor}
    Let $p/q\in\mathbb{Q}\setminus\mathbb{Z}$. Then we have the following.            
\begin{enumerate}[(i)]
    \item Suppose $1, \gamma, \delta$ linearly independent over $\mathbb Q$. Then \eqref{cccond} holds for any choice of coprime $p$ and $q$.
    \item Suppose $\gamma$ and $\delta$ irrational, there exist integers $j_1$, $j_2$ and $j_3$ with no common divisors such that $j_1$ and $j_2$ have a nontrivial common divisor, and $j_1\gamma+j_2\delta+j_3=0$. Then \eqref{cccond} holds for any choice of coprime $p$ and $q$.
    \item Suppose $\gamma$ and $\delta$ irrational, there exist integers $j_1$, $j_2$ and $j_3$  such that $j_1$ and $j_2$ have no nontrivial common divisors and $j_1\gamma+j_2\delta+j_3=0$. Then \eqref{cccond} holds for any choice of coprime $p$ and $q$, except for $p/q=j_1/j_2$.
    \item Suppose $\gamma$ rational and $\delta$ irrational, or viceversa. Then \eqref{cccond} holds for any choice of coprime $p$ and $q$.
    \item Suppose $\gamma$ and $\delta$ are integers. Then \eqref{cccond} does not hold for any choice of coprime $p$ and $q$.
    \item Suppose that both $\gamma$ and $\delta$ are rational, but at least one of them is not integer. Then there exist coprime $p$ and $q$ such that \eqref{cccond} is achieved. 
\end{enumerate}
\end{proposition}

\begin{proof}
If $1, \gamma, \delta$ are linearly independent over the rationals, then $\gamma p+\delta q$ is not an integer and therefore condition \eqref{cccond} is achieved for any choice of $p/q$.

Let us now study the case $1,\gamma,\delta$ linearly dependent over the rationals. In particular, there exist three integers $j_1,j_2,j_3$, not all equal to $0$, such that
\[
j_1\gamma+j_2\delta+j_3=0.
\]
Assume without loss of generality that $j_1$, $j_2$ and $j_3$ do not have a nontrivial common divisor.

Suppose first that both $\gamma$ and $\delta$ are irrational. Then $j_1\neq0$ and $j_2\neq0$, and
\[
\delta=-\frac{j_1}{j_2}\gamma-\frac{j_3}{j_2}
\] 
so that
\[
\gamma p+\delta q=\gamma p+\left(-\frac{j_1}{j_2}\gamma-\frac{j_3}{j_2}\right)q=\left(\frac pq-\frac{j_1}{j_2}\right)q\gamma-\frac{j_3}{j_2}q.
\]
If $p/q\neq j_1/j_2$, then $\gamma p+\delta q$ is not rational and \eqref{cccond} is achieved.

Assume now $p/q=j_1/j_2$.

Suppose first that $j_1$ and $j_2$ are coprime. Then $q=\pm j_2$ and since
\[
\gamma p+\delta q=-\frac{j_3}{j_2}q=\pm{j_3},
\]
it follows that \eqref{cccond} is not achieved. 

Suppose now that $j_1=hp$ and $j_2=hq$ for some integer  $h\neq\pm1$. Of course $h$ does not divide $j_3$. Then 
\[
\gamma p+\delta q=-\frac{j_3}{j_2}q=-\frac{j_3}{h}\notin\mathbb Z,
\]
and \eqref{cccond} is achieved.

Suppose now $\gamma$ irrational and $\delta$ rational.  Since $p\neq 0$, then $\gamma p+\delta q$ is not rational and condition \eqref{cccond} is achieved.

Similarly, if $\delta$ is irrational and $\gamma$ is rational,  since $q\neq 0$, then $\gamma p+\delta q$ is not rational and condition \eqref{cccond} is achieved.

If $\gamma,\delta\in\mathbb Z$, then \eqref{cccond} is not achieved for any value of $p$ and $q$.

Finally, if $\gamma$ and $\delta$ are both rational, but, say, $\delta$ is not integer, then it suffices to let $p$ be any integer such that $\gamma p\in\mathbb Z$, and $q$ any integer prime with $p$ such that $\delta q\notin \mathbb Z$.  
\end{proof}

Since the Gegenbauer polynomials are particular cases of the Jacobi polynomials,
\[
P^{\lambda}_m(x)=\frac{\Gamma\big(\lambda+\frac12\big)\Gamma(m+2\lambda)}{\Gamma(2\lambda)\Gamma\big(m+\lambda+\frac12\big)}P^{\lambda-\frac12,\lambda-\frac12}_m(x)
\]
($\lambda>-1/2$), we can recover the result in \cite{BMS} on $\lambda$-gegenbadly approximable numbers.
\begin{Cor}\label{gegenbadly}
Let $\lambda>-1/2$. There exist $r\in\mathbb R\setminus \pi\mathbb Z$ and positive constants $C$ and $m_{0}$ such
that for $m\geq m_{0}$,%
\[
\left|P_m^{\lambda}(\cos r)\right|\geq Cm^{\lambda-1}
\]
if and only if $r/\pi=p/q\in\mathbb Q\setminus\mathbb Z$, with $p$ and $q$ coprime, and 
\begin{equation}
\label{ccccond}
\lambda p-\frac{\lambda-1}2 q\notin\mathbb Z.
\end{equation}
In particular, the following two cases are quickly treated.                         
\begin{enumerate}[(i)]
    \item Suppose $\lambda$ irrational. Then \eqref{ccccond} holds for any choice of coprime $p$ and $q$, except for $p/q=1/2$.
    \item Suppose $\lambda$ is an odd integer. Then \eqref{ccccond} does not hold for any choice of coprime $p$ and $q$.
    \item Suppose that $\lambda\in\mathbb Q$, but $\lambda$ is not an odd integer. Then there exist coprime $p$ and $q$ such that \eqref{ccccond} is achieved.
\end{enumerate}
\end{Cor}

\begin{proof}
It suffices to apply Theorem \ref{jacobadly} and Proposition \ref{jacobcor}, and notice that the cases (i), (ii) and (iv) of Proposition \ref{jacobcor} do not occur.
\end{proof}

We are now ready to apply Theorem \ref{jacobadly} and Proposition \ref{jacobcor} to estimate \eqref{int} from below.

\begin{lemma}
\label{Lemma un cerchio}
Let $0<r<\pi$. There exist positive constants $C$ and $m_{0}$ such
that for $m\geq m_{0}$,%
\[
\left\vert \int_{B_{r}(o)}Z_{o}^{m}(x)d\mu(x)\right\vert \geq Cd_{m}%
m^{-a-3/2}
\]
if and only if $r/\pi=p/q$, where $p$ and $q$ are coprime and
\[
\frac{d+d_0+2}4p-\frac{d-1}4q\notin\mathbb Z.
\]
In particular, in this case, $d\not \equiv 1\,( \operatorname{mod}4)$.
\end{lemma}
Observe that the above condition can easily be satisfied, for example take $p\in 4\mathbb Z$ and $q$ odd.

\begin{proof}
By Lemma \ref{Trasf bolla}, for any $0\leq r\leq\pi$ and for any
$m\geq1$ we have%
\[
\int_{B_{r}(o)}Z_{o}^{m}(x)d\mu(x)=\frac{c( a,b) d_{m}}{mP_{m}^{a,b}%
(1)}P_{m-1}^{a+1,b+1}(\cos r)\left(  \sin  \frac{r}{2}  \right)
^{2a+2}\left(  \cos \frac{r}{2}  \right)  ^{2b+2}.
\]
We can now apply Theorem \ref{jacobadly} with $\alpha=a+1=d/2$ and $\beta=b+1=d_0/2$, so that
$\gamma=(\alpha+\beta+1)/2=(d+d_0+2)/4$ and $\delta=-(2\alpha-1)/4=-(d-1)/4$. The result now follows immediately, recalling that $P_m^{a,b}(1)={m+a \choose m}\approx m^a$.
\end{proof}

 We can now prove Theorem \ref{Thm 1} and Theorem \ref{d1mod4}.

\begin{proof}
[Proof of Theorem \ref{Thm 1}]Let $m_{0}$ as in Lemma \ref{Lemma un cerchio}
and let $L\geq m_{0}$. Then by Lemma \ref{Lemma Quadrato Discrep},
Proposition \ref{Prop Cassels} and Lemma \ref{Lemma un cerchio}, we have%
\begin{align*}
&  \left\Vert D_{p\pi/q}\right\Vert _{2}^{2}\\
  & =  \ \sum_{m=1}^{+\infty} d_{m}^{-2}  \left\vert \int_{B_{p\pi/q}(o)}%
Z_{o}^{m}(x)d\mu(x)\right\vert ^{2} \sum_{\ell=1}^{d_{m}}\left\vert
\sum_{j=1}^{N}a_{j}Y_{m}^{\ell}(x_{j})\right\vert ^{2}\\
 & \geq  \ \min_{m_{0}\leq m\leq L}\left(  d_{m}^{-2}
\left\vert \int_{B_{p\pi/q}(o)}Z_{o}^{m}(x)d\mu(x)\right\vert ^{2}  \right) 
\sum_{m=m_{0}}^{L}\sum_{\ell=1}^{d_{m}}\left\vert \sum_{j=1}^{N}%
a_{j}Y_{m}^{\ell}(x_{j})\right\vert ^{2}\\
 & \geq  \ C\left(  \min_{m_{0}\leq m\leq L}m^{-2a-3}\right)
\left(  C_{1}L^{d}\sum_{j=1}^{N}a_{j}^{2}-C_{0}m_0^d\right)  \geq
CL^{-2a-3}\left(  C_{1}L^{d}\sum_{j=1}^{N}a_{j}^{2}-C_{0}m_0^d\right)  .
\end{align*}
Applying the Cauchy-Schwarz inequality to $\sum_{j=1}^{N}a_{j}=1$ gives
\[
\sum_{j=1}^{N}a_{j}^{2}\geq\frac{1}{N}.
\]
Let $N\geq N_{0}=C_{1}/(2C_{0})$. Then, setting $L=\lfloor
m_0(2C_{0}C_{1}^{-1}N)^{1/d}\rfloor  +1$, we have $L\geq m_{0}$, \[C_{1}%
L^{d}\sum_{j=1}^{N}a_{j}^{2}-C_{0}m_0^d\geq C_1\frac{L^d}{2N}\] and%
\begin{equation}
\left\Vert D_{p\pi/q}\right\Vert _{2}^{2}%
\geq CN^{-1-\frac{1}{d}}. \label{Discrep N}%
\end{equation}
Let now $N<N_{0}$ and let us consider the points and weights%
\[
\widetilde{x}_{j}=%
\begin{cases}
x_{j} & 1\leq j\leq N-1,\\
x_{N} & N\leq j\leq N_{0},
\end{cases}
\qquad\widetilde{a}_{j}=%
\begin{cases}
a_{j} & 1\leq j\leq N-1,\\
\dfrac{a_{N}}{N_{0}-N+1} & N\leq j\leq N_{0}.
\end{cases}
\]
Since the discrepancy $\widetilde{D}_{r}$ of the points $\{\widetilde{x}%
_{j}\}_{j=1}^{N_{0}}$ and weights $\{\widetilde{a}_{j}\}_{j=1}^{N_{0}}$
coincides with the discrepancy $D_{r}$ of the points $\{x_{j}\}_{j=1}^{N}$ and
weights $\{a_{j}\}_{j=1}^{N}$, applying (\ref{Discrep N}) to $\widetilde
{D}_{r}$ gives%
\[
\left\Vert D_{p\pi/q}\right\Vert _{2}^{2}\geq CN_{0}^{-1-\frac{1}{d}}\geq CN_{0}
^{-1-\frac{1}{d}}N^{-1-\frac{1}{d}}%
\]
also when $1\leq N<N_{0}.$
\end{proof}

\begin{proof}[Proof of Theorem \ref{d1mod4}] Let $H$ and $L$ be positive integers that will be fixed later. By Lemma \ref{Lemma Quadrato Discrep}, 
\begin{align*}
\sum_{n=1}^H \frac1H\left\Vert D_{p_n\pi/q_n}\right\Vert _{2}^{2}&=\sum
_{m=1}^{+\infty}\sum_{\ell=1}^{d_{m}}\left\vert \sum_{j=1}^{N}a_{j}Y_{m}%
^{\ell}( x_{j}) \right\vert ^{2}\sum_{n=1}^H\frac{1}{H} \left\vert \int_{B_{p_n\pi/q_n}( o) }%
d_{m}^{-1}Z_{o}^{m}( y) d\mu( y) \right\vert ^{2}\\
&\geq \sum
_{m=1}^{L}\sum_{\ell=1}^{d_{m}}\left\vert \sum_{j=1}^{N}a_{j}Y_{m}%
^{\ell}( x_{j}) \right\vert ^{2}\sum_{n=1}^H\frac{1}{H} \left\vert \int_{B_{p_n\pi/q_n}( o) }%
d_{m}^{-1}Z_{o}^{m}( y) d\mu( y) \right\vert ^{2}.
\end{align*}
By Lemma \ref{Trasf bolla} and \cite[Theorem 8.21.8]{szego}, uniformly in $\varepsilon\leq r\leq\pi-\varepsilon$,
\begin{align*}
&\int_{B_{r}(o)}d_{m}^{-1}Z_{o}^{m}(x)d\mu(x)=\frac{c( a,b)}{mP_{m}^{a,b}%
(1)}P_{m-1}^{a+1,b+1}(\cos r)\left(  \sin \frac{r}{2}  \right)
^{2a+2}\left(  \cos  \frac{r}{2}  \right)  ^{2b+2}\\
&= \frac{c( a,b)\pi^{-\frac{1}{2}}}{m^{\frac32}P_{m}^{a,b}%
(1)}\left(\sin \frac{r}{2} \right)^{a+\frac{1}{2}}\left(\cos\frac{r}{2} \right)^{b+\frac{1}{2}}\cos\left(\Big(m+\frac{a+b+1}{2} \Big)r-\Big(a+\frac{3}{2}\Big)\frac{\pi}{2} \right)+O(m^{-a-\frac{5}{2}}).
\end{align*}
Since $d=4S+1$ for some positive integer $S$, then $\mathcal{M}$ is either $\mathcal{S}^d$ or $P^d(\mathbb{R})$ and therefore $d_0=d$ or $d_0=1$. Also $a=(d-2)/2$ and $b=(d_0-2)/2$. Hence, whenever $m+S+(d_0-1)/4\notin q_n\mathbb{Z}$, since $q_n\leq 2n\log(n+1)$ (see \cite[formula (3.13)]{BS}), we have
\begin{align*}
&\left|\cos\left(\Big(m+\frac{d+d_0-2}{4} \Big)\frac{p_n}{q_n}\pi-\frac{d+1}{4}\pi \right)\right|=\left|\cos\left(\Big(m+\frac{4S+d_0-1}{4} \Big)\frac{p_n}{q_n}\pi-\frac{4S+2}{4}\pi \right)\right|\\
&=\left|\sin\left(\Big(m+S+\frac{d_0-1}{4} \Big)\frac{p_n}{q_n}\pi\right)\right|\geq \left|\sin\left(\frac{\pi}{q_n}\right)\right|\geq \frac{2}{q_n}\geq \frac{1}{n\log (n+1)}.
\end{align*}
Now, for each $m$ such that  \[m+S+\frac{d_0-1}{4}< q_1\cdots q_H,\] there exists $1\leq n\leq H$ such that $m+S+(d_0-1)/4\notin q_n\mathbb{Z}$. Hence, if \[L<q_1\cdots q_H-S-\frac{d_0-1}{4},\] then, for $m\le L$,
\begin{align*}\sum_{n=1}^H\frac{1}{H} \left\vert \int_{B_{p_n\pi/q_n}( o) }%
d_{m}^{-1}Z_{o}^{m}( y) d\mu( y) \right\vert ^{2}\geq \frac{1}{H}\left(\frac{1}{H^2\log^2 H}  \frac{c_1}{m^{2a+3}}-\frac{c_2}{m^{2a+5}}\right),
\end{align*}
and, for $m\geq c_3H\log H$ with $c_3$ a suitable constant (say, $c_3=(2c_2/c_1)^{1/2}$),
\[\sum_{n=1}^H\frac{1}{H} \left\vert \int_{B_{p_n\pi/q_n}( o) }%
d_{m}^{-1}Z_{o}^{m}( y) d\mu( y) \right\vert ^{2}\geq \frac{1}{H^3\log^2 H}\frac{c_1}{2m^{2a+3}}.\]
Therefore, 
\begin{align*}
\sum_{n=1}^H \frac1H\left\Vert D_{p_n\pi/q_n}\right\Vert _{2}^{2}&\geq \sum
_{m=1}^{L}\sum_{\ell=1}^{d_{m}}\left\vert \sum_{j=1}^{N}a_{j}Y_{m}%
^{\ell}( x_{j}) \right\vert ^{2}\sum_{n=1}^H\frac{1}{H} \left\vert \int_{B_{p_n\pi/q_n}( o) }%
d_{m}^{-1}Z_{o}^{m}( y) d\mu( y) \right\vert ^{2}\\
&\geq \sum
_{m=c_3 H\log H}^{L}\sum_{\ell=1}^{d_{m}}\left\vert \sum_{j=1}^{N}a_{j}Y_{m}%
^{\ell}( x_{j}) \right\vert ^{2}\frac{1}{H^3\log^2 H}  \frac{c_1}{2m^{2a+3}}.
\end{align*}
By Proposition \ref{Prop Cassels},
%\[\sum
%_{m=0}^{cH\log H}\sum_{\ell=1}^{d_{m}}\left\vert \sum_{j=1}^{N}a_{j}Y_{m}%
%^{\ell}( x_{j}) \right\vert ^{2}=\sum
%_{m=0}^{cH\log H}\sum_{j,k=1}^{N}a_{j}a_k Z^m_{x_j}(x_k)\leq \sum
%_{m=0}^{cH\log H} d_m\leq cH^d\log^d H.\]
\begin{align*}
\sum_{n=1}^H \frac1H\left\Vert D_{p_n\pi/q_n}\right\Vert _{2}^{2}&\geq\frac{1}{H^3\log^2 H}\frac{c_1}{2L^{2a+3}}\sum
_{m=c_3H\log H}^{L}\sum_{\ell=1}^{d_{m}}\left\vert \sum_{j=1}^{N}a_{j}Y_{m}%
^{\ell}( x_{j}) \right\vert ^{2} \\& \geq \frac{1}{H^3\log^2 H}\frac{c_1}{2L^{d+1}}\left(C_1\frac{L^d}{N}-C_0c_3^dH^d\log^dH \right).
\end{align*}
If $H\ge13$, by \cite[Theorem 4]{Robin}, then $q_1\cdots q_H >e^{H\log H}$. Since we need $L<q_1\cdots q_H-S-(d_0-1)/4$,  it suffices to have $S+(d_0-1)/4\le L\leq e^{H\log H}/2$. 
%\nota{We can therefore choose $H$ and $L$ large enough so that $ (H\log H)/2\leq \log L\leq H\log H-\log 3$.} 
Let us choose
%\[
%L= L_N:=\kappa N^{1/d}\log N(\log\log N)^\varepsilon,\quad H= H_L:=\frac{\log(2L)}{(\log\log(2L))^{1-\varepsilon}}\leq \frac{\log N}{(\log\log N)^{1-\varepsilon}},\]
%where $\kappa$ is a constant to be chosen later and $\varepsilon$ is any positive constant. With this choice, indeed
%\[
%e^{H\log H}= e^{\frac{\log(2L)}{(\log\log(2L))^{1-\varepsilon}}\log\left(\frac{\log(2L)}{(\log\log(2L))^{1-\varepsilon}}\right)}=
%e^{{\log(2L)}(\log\log (2L))^\varepsilon\left(1-\frac{(1-\varepsilon)\log\log\log(2L)}{(\log\log(2L))}\right)}\ge 2L
%\]
%for $L$ large enough (say, $\log\log(2L)\ge (e/(e-1))^{1/\varepsilon}$). 
%\nota{Giacomo} If we instead set
\[
L= L_N:=\kappa N^{1/d}\log N,\quad H= H_L:=\tau\frac{\log(2L)}{\log\log(2L)}= \tau \frac{\frac 1 d\log N+\log(2\kappa)+\log\log N}{\log(\frac 1 d\log N+\log(2\kappa)+\log\log N)},\]
for some $\kappa,\tau\ge1$ to be fixed later. Then
\[
e^{H\log H}= e^{\frac{\tau\log(2L)}{\log\log(2L)}\log\left(\frac{\tau\log(2L)}{\log\log(2L)}\right)}=
e^{{\tau\log(2L)}\left(1+\frac{\log\tau}{\log\log(2L)}-\frac{\log\log\log(2L)}{\log\log(2L)}\right)}\ge 2L
\]
for all $L> e/2$ if $\tau$ is large enough (say, $\tau\ge2$). 
Also, if $N\ge e$ and $2L\geq e^e$,
\begin{align*}
C_0c_3^dH^d\log^dH&=C_0c_3^d\left(\tau\frac{\log(2L)}{\log\log(2L)}\log\left(\tau\frac{\log(2L)}{\log\log(2L)}\right)\right)^d\\
&=C_0c_3^d\left(\tau{\log(2L)}\left(\frac{\log\tau}{{\log\log(2L)}}+1-\frac{{\log\log\log(2L)}}{{\log\log(2L)}}\right)\right)^d\\
&\le C_0c_3^d\tau^d{\left(\log(2\kappa)+\log\log N+\frac 1 d \log N\right)^d}(\log\tau+1)^d \\
&\le C_0c_3^d\tau^d(\log\tau+1)^d\left({\log(2\kappa)}+1+\frac 1 d\right)^d \log^d N \\
&\leq \frac{C_1}{2}{\kappa^d\log^d N}= \frac{C_1}{2}\frac{L^d}{N},
\end{align*}
if 
\[
C_0^{1/d}c_3\tau(\log\tau+1)\left({\log(2\kappa)}+1+\frac 1 d\right) \leq \left(\frac{C_1}{2}\right)^{\frac1d}\kappa.
\]
If $2\kappa\geq e^2$, then
\[
\log(2\kappa)+1+\frac 1 d\leq2\log(2\kappa),
\]
so that we need
\[
\frac{\kappa}{2\log(2\kappa)}\ge\frac{(2C_0)^{1/d}c_3\tau(\log\tau+1)}{C_1^{1/d}}=:\gamma.
\]
Assuming, as we may, $\gamma\ge e/4$, 
it suffices
\[
\kappa\ge4\gamma\log(4\gamma).
\]
Hence, if $\tau=13$ and $\kappa\ge\max\{S+(d_0-1)/4,\,e^e/2,\,4\gamma\log(4\gamma)\}$, then for all $N\geq 3$,
\[\sum_{n=1}^H \frac1H\left\Vert D_{p_n\pi/q_n}\right\Vert _{2}^{2}\geq C_1c_1\frac{1}{H^3\log^2 H}\frac{1}{L}\frac{1}{4N}\geq CN^{-1-1/d}\frac{\log\log N}{\log^4N}.\]
Since the function $x\mapsto x/\log(x)$ is increasing in $x\ge e$, it follows that
\[
H_L\leq \tau\frac{(2+\log(2\kappa)) \log N}{\log\left( (2+\log(2\kappa)\log N\right)}
\leq
\tau(2+\log(2\kappa)) \frac{\log N}{\log\left( \log N\right)}.
\]
Thus, in the statement of the theorem, we can say that for all $N\geq 3$, \[n\le \tau(2+\log(2\kappa)) \frac{\log N}{\log\left( \log N\right)}.\]

%\noindent($C_0$ si può stimare facilmente)
%
%\noindent Il pezzo per $N$ piccolo a questo punto non serve.
%
%\nota{fine Giacomo}
%
%
%Also, if $\kappa$ is sufficiently large,
%\[C_0H^d\log^dH\leq \frac{C_1}{2}\frac{L^d}{N}.\]
%Hence, for sufficiently large $N$, say $N\geq N_0$,
%\[\sum_{n=1}^H \frac1H\left\Vert D_{p_n\pi/q_n}\right\Vert _{2}^{2}\geq C_1c_1\frac{1}{H^3\log^2 H}\frac{1}{L}\frac{1}{2N}\geq CN^{-1-1/d}\frac{(\log\log N)^{1-4\varepsilon}}{\log^4N}.\]
%Let now $N<N_{0}$ and let us consider the points and weights%
%\[
%\widetilde{x}_{j}=%
%\begin{cases}
%x_{j} & 1\leqslant j\leqslant N-1,\\
%x_{N} & N\leqslant j\leqslant N_{0},
%\end{cases}
%\qquad\widetilde{a}_{j}=%
%\begin{cases}
%a_{j} & 1\leqslant j\leqslant N-1,\\
%\dfrac{a_{N}}{N_{0}-N+1} & N\leqslant %j\leqslant N_{0}.
%\end{cases}
%\]
%Since the discrepancy $\widetilde{D}_{r}$ of the points $\{\widetilde{x}%
%_{j}\}_{j=1}^{N_{0}}$ and weights $\{\widetilde{a}_{j}\}_{j=1}^{N_{0}}$
%coincides with the discrepancy $D_{r}$ of the points $\{x_{j}\}_{j=1}^{N}$ and
%weights $\{a_{j}\}_{j=1}^{N}$, 
%\begin{align*}
%\sum_{n=1}^H \frac1H\left\Vert D_{p_n\pi/q_n}\right\Vert _{2}^{2}&\geq CN_0^{-1-1/d}\frac{(\log\log N_0)^{1-4\varepsilon}}{\log^4N_0}
%\\
%&\geq CN_0^{-1-1/d}\frac{(\log\log N_0)^{1-4\varepsilon}}{\log^4N_0} N^{-1-1/d}\frac{(\log\log N)^{1-4\varepsilon}}{\log^4N},
%\end{align*}
%also when $3\leq N<N_{0}.$
\end{proof}

\begin{comment}
\nota{Giacomo} Possiamo prendere 
\[
L= L_N:=CN^{1/d}\log N(\log\log N)^\varepsilon,\quad H= H_L:=\frac{\log(3L)}{(\log\log(3L))^{1-\varepsilon}}\approx \frac{\log N}{(\log\log N)^{1-\varepsilon}},\] e allora 
\[\sum_{n=1}^H \frac1H\left\Vert D_{p_n\pi/q_n}\right\Vert _{2}^{2}\geq C_1c_1\frac{1}{H^3\log^2 H}\frac{1}{L}\frac{1}{2N}\geq 
C\frac{(\log\log N)^{1-4\varepsilon}}{N^{1+1/d}\log^4N}.\]
Volendo, le scelte ottimali sarebbero
\[
L\ge \widetilde L_N:=e^{-W_{-1}(-CN^{-1/d})}, \quad H\ge \widetilde H_L:=e^{W_0(\log(3L))}. 
\]
Le stime che si trovano su Wikipedia danno
\[
\widetilde L_N\le CN^{1/d}e^{\sqrt{\log(CN^{2/d})}},\quad \widetilde H_L\le\frac{\log(3L)}{\log\log(3L)}e^{\frac{e}{e-1}\frac{\log\log\log(3L)}{\log\log(3L)}}\le H_L.
\]
In realtà, la prima stima è troppo grossolana in quanto è una formula che non è stata pensata per essere fine per $N$ grande: in effetti $L_N$ è più piccolo. La seconda stima fornisce qualcosa di meglio di $H_L$, ma non troppo.

\nota{versione vecchia:}
If $H$ is sufficiently large, then $q_1\cdots q_H \geq e^H$ \nota{mettere referenza}. Since we need $L<q_1\cdots q_H-2S-(d_0-1)/4$, it suffices to consider $L<e^H$, that is $H>\log L$. If we now set $L=cN^{\frac 1d}\log^{1+\varepsilon} N$ and $H=2\log L$, then $L^d/N\gtrsim H^d\log^dH$ and 
\[\sum_{n=1}^H \frac1H\left\Vert D_{p_n\pi/q_n}\right\Vert _{2}^{2}\geq \frac{1}{H^3\log^2 H}\frac{c}{L^{d+1}}\frac{L^d}{N}\geq \frac{c}{N^{\frac 1d+1}\log^{4+\varepsilon} N}.\]
\end{comment}

\bibliographystyle{amsalpha}
\bibliography{single-bib}

\end{document}